\documentclass[a4, 12pt]{amsart}
\usepackage[mathscr]{eucal}
\usepackage{amssymb}
\usepackage{latexsym}
\usepackage{amsthm}
\theoremstyle{plain}
\newtheorem{theorem}{Theorem}[section]
\newtheorem{proposition}{Proposition}[section]

\newtheorem{remark}{Remark}[section]
\newtheorem{lemma}{Lemma}[section]
\newtheorem{example}{Example}[section]

\makeatletter
\@addtoreset{equation}{section}

\title[The Gauss image of $\lambda$-hypersurfaces and a Bernstein type problem]
{The Gauss image of $\lambda$-hypersurfaces and a Bernstein type problem }
\author{Qing-Ming Cheng and  Guoxin Wei}
\address{Qing-Ming Cheng \\  Department of Applied Mathematics, Faculty of Sciences,
Fukuoka  University, 814-0180, Fukuoka,  Japan, cheng@fukuoka-u.ac.jp}
\address{Guoxin Wei \\  School of Mathematical Sciences, South China Normal University,
510631, Guangzhou,  China, weiguoxin@tsinghua.org.cn}

\begin{document}
\maketitle

\begin{abstract}
\noindent  In this paper, our purpose is to study rigidity theorems for $\lambda$-hypersurfaces
in Euclidean space under Gauss map. As  a Bernstein type problem for $\lambda$-hypersurfaces,
we  prove that an entirely graphic $\lambda$-hypersurface in Euclidean space is a hyperplane.
\end{abstract}

\footnotetext{ 2010 \textit{ Mathematics Subject Classification}: 53C44, 53C42.}

\footnotetext{{\it Key words and phrases}: Bernstein problem, the weighted area functional,
$\lambda$-hypersurfaces, the weighted volume-preserving mean curvature flow.}

\footnotetext{The first author was partially  supported by JSPS Grant-in-Aid for Scientific Research (B): No. 24340013
and Challenging Exploratory Research No. 25610016. The second author was supported by NSFC No. 11371150.}

\section {Introduction}

\noindent
Let $X: M\rightarrow \mathbb{R}^{n+1}$ be a smooth $n$-dimensional immersed hypersurface in the $(n+1)$-dimensional
Euclidean space $\mathbb{R}^{n+1}$. In \cite{cw2}, Cheng and Wei have introduced the notation of
the weighted volume-preserving mean curvature flow, which is defined as the following:
a family $X(\cdot, t)$ of smooth immersions
$$
X(\cdot, t):M\to  \mathbb{R}^{n+1}
$$
with  $X(\cdot, 0)=X(\cdot)$ is called {\it a weighted volume-preserving mean curvature flow} if
\begin{equation}
\dfrac{\partial X(t)}{\partial t}=-\alpha(t) N(t) +\mathbf{H}(t)
\end{equation}
holds, where
$$
\alpha(t) =\dfrac{\int_MH(t)\langle N(t), N\rangle e^{-\frac{|X|^2}2}d\mu}{\int_M\langle N(t), N\rangle e^{-\frac{|X|^2}2}d\mu},
$$
 $\mathbf{H}(t)=\mathbf{H}(\cdot,t)$ and $N(t)$ denote the mean curvature vector  and the unit normal vector of hypersurface
 $M_t=X(M^n,t)$ at point $X(\cdot, t)$, respectively  and   $N$ is the unit normal vector of $X:M\to  \mathbb{R}^{n+1}$.
One  can prove  that the flow (1.1) preserves the weighted volume $V(t)$ defined by
$$
V(t)=\int_M\langle X(t),N\rangle e^{-\frac{|X|^2}{2}}d\mu.
$$
\noindent
{\it The weighted area functional}
$A:(-\varepsilon,\varepsilon)\rightarrow\mathbb{R}$ is defined by
$$
A(t)=\int_Me^{-\frac{|X(t)|^2}{2}}d\mu_t,
$$
where $d\mu_t$ is the area element of $M$ in the metric induced by $X(t)$.
Let $X(t):M\rightarrow \mathbb{R}^{n+1}$ with $X(0)=X$ be a  variation of $X$.
If $V(t)$ is constant for any $t$, we call  $X(t):M\rightarrow \mathbb{R}^{n+1}$ {\it  a weighted volume-preserving variation of $X$}.
Cheng and Wei \cite{cw2} have proved that  $X:M\rightarrow \mathbb{R}^{n+1}$ is a critical point of the weighted area functional $A(t)$
for all weighted volume-preserving variations if and only  if there exists constant $\lambda$ such that
\begin{equation}\label{eq:11.1}
\langle X, N\rangle +H=\lambda.
\end{equation}
An immersed hypersurface $X(t):M\rightarrow \mathbb{R}^{n+1}$ is called {\it a $\lambda$-hypersurface} if the equation \eqref{eq:11.1} is satisfied.

\begin{remark}
When $\lambda=0$,  the $\lambda$-hypersurface becomes   a self-shrinker of the mean curvature flow.
\end{remark}

\begin{example}
The $n$-dimensional Euclidean space $\mathbb{R}^{n}$ is a complete and non-compact $\lambda$-hypersurface
in $\mathbb{R}^{n+1}$ with $\lambda=0$.
\end{example}
\begin{example}
The $n$-dimensional sphere $S^n(r)$ with radius $r>0$ is a compact $\lambda$-hypersurface in $\mathbb{R}^{n+1}$
with $\lambda=\frac{n}r-r$.
\end{example}
\begin{example}
For $1\leq k\leq n-1$, the $n$-dimensional cylinder  $S^k(r)\times \mathbb{R}^{n-k}$ with radius $r>0$ is a complete
and non-compact $\lambda$-hypersurface in $\mathbb{R}^{n+1}$ with $\lambda=\frac{k}r-r$.
\end{example}

\noindent
It is well know that the Gauss map of  hypersurfaces in $\mathbb{R}^{n+1}$ plays a very important role in study of  hypersurfaces.
For constant mean curvature surfaces in  $\mathbb{R}^{3}$, a beautiful result of Hoffeman, Osserman and Schoen \cite{HOS} shows
that the plane and the right circular cylinder are the only complete surfaces with constant mean curvature in $\mathbb{R}^{3}$,
of which the image under Gauss map lies a closed hemisphere.

\noindent
Our purpose in this paper is to study  $\lambda$-hypersurfaces by Gauss map.  We want to attack the following problem:

\noindent
{\bf   Problem}.
Let $X:M\rightarrow \mathbb{R}^{n+1}$ be an $n$-dimensional complete $\lambda$-hypersurface in the Euclidean space $\mathbb{R}^{n+1}$.
If the image under the Gauss map is contained in an open hemisphere, then is $X:M\rightarrow \mathbb{R}^{n+1}$  a hyperplane? If the image under the Gauss map is contained in a closed hemisphere, then is $X:M\rightarrow \mathbb{R}^{n+1}$  a hyperplane or a cylinder whose cross section is an $(n-1)$-dimensional $\lambda$-hypersurface in $\mathbb{R}^{n}$?

\noindent
For the above  problem, we solve it under the assumption of proper.

\begin{theorem}\label{theorem 1}
Let $X:M\rightarrow \mathbb{R}^{n+1}$ be an $n$-dimensional complete properly $\lambda$-hypersurface in the Euclidean space $\mathbb{R}^{n+1}$. If the image under the Gauss map is contained in an open hemisphere, then $X:M\rightarrow \mathbb{R}^{n+1}$ is a hyperplane. If the image under the Gauss map is contained in a closed hemisphere, then $X:M\rightarrow \mathbb{R}^{n+1}$ is a hyperplane or a cylinder whose cross section is an $(n-1)$-dimensional $\lambda$-hypersurface in $\mathbb{R}^{n}$.
\end{theorem}

\begin{theorem}\label{theorem 2}
Let $X:M\rightarrow \mathbb{R}^{n+1}$ be an $n$-dimensional complete properly $\lambda$-hypersurface in the Euclidean space $\mathbb{R}^{n+1}$. If the image under the Gauss map is contained in $S^n\setminus \overline{S}^{n-1}_{+}$, then $X:M\rightarrow \mathbb{R}^{n+1}$ is a hyperplane.
\end{theorem}

\begin{remark}
 The set $(S^n\setminus \overline{S}_{+}^{n-1})\cup \{p\}$ contains a great circle, here $p\in \overline{S}_{+}^{n-1}$, and the nontrivial $\lambda$-hypersurface $S^1\times \mathbb{R}^{n-1}\subset \mathbb{R}^{n+1}$ whose Gauss image is a great circle. Hence the Gauss image restriction in the theorem \ref{theorem 3} is optimal.
\end{remark}

\begin{remark}
Since  a $\lambda$-hypersurface in the Euclidean space $\mathbb{R}^{n+1}$ is a self-shrinker when $\lambda=0$,
we should remark that Ding, Xin and Yang \cite{[DXY]} have proved the same results for complete proper
self-shrinkers  in $\mathbb{R}^{n+1}$.
\end{remark}

\noindent
Furthermore, since the image of an  entire graphic $\lambda$-hypersurface under Gauss map  is contained in an open hemisphere, we  prove that
the assertion in the  problem is true for entire graphic $\lambda$-hypersurfaces.

\begin{theorem}\label{theorem 3}
Let $X:M\rightarrow \mathbb{R}^{n+1}$ be an $n$-dimensional entire graphic $\lambda$-hypersurface in the Euclidean space $\mathbb{R}^{n+1}$.
Then  $X:M\rightarrow \mathbb{R}^{n+1}$
is a hyperplane $\mathbb{R}^{n}$.
\end{theorem}

\begin{remark}
In the case of $\lambda=0$, that is, in the case of self-shrinkers,  Ecker and Huisken \cite{[EH]} proved
that $X:M\rightarrow \mathbb{R}^{n+1}$ is a hyperplane if it is an entire graphic self-shrinker
with polynomial area growth in $\mathbb{R}^{n+1}$. Recently, Wang \cite{[W]} removed the assumption of polynomial area growth {\rm (}cf. Ding and Wang
\cite{DW}{\rm )}.
\end{remark}

\vskip3mm
\noindent
{\bf Acknowledgement.}
Authors  would like to express their gratitude  to
Professor H. Li  for invaluable discussions.

\section{Proof of Theorem \ref{theorem 3}}
\noindent
Letting $X:M\rightarrow \mathbb{R}^{n+1}$ be an $n$-dimensional entire graphic hypersurface in the Euclidean space $\mathbb{R}^{n+1}$, then we can write
\begin{equation}
X=(x_1,\cdots, x_n,f)=(x,f),
\end{equation}
where $x=(x_1,\cdots, x_n), f=f(x_1,\cdots, x_n)$.
Denote the orthonormal basis of  $\mathbb{R}^{n+1}$ by $\{ E_1, E_2, \cdots, E_{n+1}\}$. We know that tangent vectors of $X:M\rightarrow \mathbb{R}^{n+1}$
are given by
$$
e_i=E_i+f_iE_{n+1},
$$
for $i=1, 2, \cdots, n$, where $f_i=\frac{\partial f}{\partial x_i}$. The induced metric on $M$ is given by
\begin{equation}
g_{ij}=\langle e_i,e_j\rangle=\delta_{ij}+f_if_j,
\end{equation}
where $\langle \cdot,\cdot\rangle$ is the canonical inner product in $\mathbb{R}^{n+1}$. The unit normal vector $N$ is given by
\begin{equation}
N=\frac{1}{\sqrt{1+|Df|^2}}\bigl(-\sum_i f_iE_i+E_{n+1}\bigl),
\end{equation}
where $Df=(f_1,\cdots,f_n)$ and $|Df|^2=f_1^2+\cdots+f_n^2$.
The mean curvature $H$ of $M$ is given by
\begin{equation}\label{eq:10}
H=\sum_{i,j}g^{ij}\langle f_{ij}E_{n+1},N\rangle=\sum_{i,j}\frac{g^{ij}f_{ij}}{\sqrt{1+|Df|^2}},
\end{equation}
where $f_{ij}=\frac{\partial^2 f}{\partial x_i\partial x_j}$.
On the other hand,
\begin{equation}\label{eq:11}
\langle X,N\rangle=\frac{f-\sum_ix_if_i}{\sqrt{1+|Df|^2}}.
\end{equation}
From \eqref{eq:11.1}, \eqref{eq:10} and \eqref{eq:11}, we get
\begin{equation}
\sum_{i,j}\frac{g^{ij}f_{ij}}{\sqrt{1+|Df|^2}}+\frac{f-\sum_ix_if_i}{\sqrt{1+|Df|^2}}=\lambda,
\end{equation}
that is,
\begin{equation}\label{eq:2.3}
\sum_{i,j}g^{ij}f_{ij}=-f+\sum_ix_if_i+\lambda \sqrt{1+|Df|^2}.
\end{equation}
We shall consider a differential  operator $L_{(\lambda,   f)}$ defined by
\begin{equation}
L_{(\lambda, f)}\psi=\sum_{i,j}a^{ij}\psi_{ij}-\langle x,D\psi\rangle-\lambda\frac{\langle Df,D\psi\rangle}
{\sqrt{1+|Df|^2}},
\end{equation}
where $\lambda$ is constant, $f$ and $\psi$  are functions on $\mathbb{R}^{n}$, $(a^{ij})$ is the inverse of a positive definite matrix
$(a_{ij})$, $x=(x_1,\cdots,x_n)\in \mathbb{R}^{n}$, $\psi_i=\frac{\partial \psi}{\partial x_i}$, $D\psi=(\psi_1,\cdots,\psi_n)$, $\psi_{ij}=\frac{\partial^2 \psi}{\partial x_i\partial x_j}$. The following proposition  is very important  in order to prove the theorem \ref{theorem 3}.

\begin{proposition}\label{lemma 1}  Assume  the minimum eigenvalue $\mu(x)$ of the matrix $(a_{ij})$ at $x\in M$ satisfies
\begin{equation}\label{eq:2.1}
\liminf_{|x|\rightarrow +\infty}\mu(x)(|x|^2-|\lambda||x|)>n.
\end{equation}
For a $C^2$-function $\psi$, if there is a positive constant $\varepsilon$  such that
\begin{equation}
L_{(\lambda, f)}\psi\geq \varepsilon\sum_{i,j}a^{ij}\psi_i\psi_j,
\end{equation}
then $\psi$ is constant.
\end{proposition}
\begin{proof}
From the condition \eqref{eq:2.1}, there is a real number $R_1$ such that if $x\in (\mathbb{R}^{n}\setminus B_{R_1})$, then
\begin{equation}
\sum_i a^{ii}<|x|^2-|\lambda||x|,
\end{equation}
where $B_{R_1}$ is an open ball of radius $R_1$ centered at origin.
We consider a function $h(x)$ on $\mathbb{R}^{n}$ defined by
\begin{equation}\label{eq:1.1}
 h(x)= {\begin{cases}
     \ \ 1,& \ \  |x|\leq R_0 \\
    \  -t(|x|^2-R_0^2)+1 , & \ \ |x|\geq R_0,
     \end{cases}}
  \end{equation}
where $R_0\geq R_1$ and $t$,  $0<t<1$,   are constant.  We know that  the function $h e^{C\psi}$  attains its maximum at a point $p \in \{x\in \mathbb{R}^{n}: h>0\}$, where $C<\varepsilon$ is a positive constant.
If  $\psi$ is not constant in $\mathbb{R}^{n}$, then there is a ball $B_{R_0}$ with $R_0\geq R_1$ such that $\psi$ is not constant in $B_{R_0}$. Suppose  the function $\psi$ attains its maximum at a point  $q\in B_{R_0}$.  We obtain from the strong maximum principle that $\psi$ is constant
since $L_{(\lambda, f)}$ is a linear  elliptic operator and $L_{(\lambda, f)}\psi\geq0$. It is  a contradiction. Hence, $\psi$ attains its maximum  only on the boundary $\partial B_{R_0}$. By the same assertion, in
$B_{\sqrt{R_0^2+1}}$, $\psi$ attains its maximum  only on the boundary $\partial B_{\sqrt{R_0^2+1}}$. We can assume
$\max\limits_{\overline{B}_{R_0}}\psi=\psi(p_1)$ and
$\max\limits_{\overline{B}_{\sqrt{R_0^2+1}}}\psi=\psi(p_2)$ where  $p_1 \in \partial B_{R_0}$   and $p_2 \in \partial B_{\sqrt{R_0^2+1}}$.
Then we have $\psi(p_1)<\psi(p_2)$. Therefore, as long as we choose $t$ sufficiently small, we obtain
\begin{equation}
(h e^{C\psi})(p_1)=(e^{C\psi})(p_1)<((1-t)e^{C\psi})(p_2)=(h e^{C\psi})(p_2).
\end{equation}
This means that the maximum of  $h e^{C\psi}$ can only be attained  in the set $\{x\in \mathbb{R}^{n}: |x|>R_0\geq R_1\}$.

\noindent
If $p\in\{x\in \mathbb{R}^{n}: |x|>R_0\}$, then, at point $p$, we have
\begin{equation}\label{eq:2.0}
h_i+Ch\psi_i=0,
\end{equation}
and
\begin{equation}\label{eq:2.00}
0\geq \sum_{i,j}a^{ij}(h e^{C\psi})_{ij}
\end{equation}
since $(a^{ij})$ is positive definite.
By a direct calculation, we have, from the definition of $h$ and   \eqref{eq:2.0},
\begin{equation}\label{eq:2.2}
\aligned
&\ \ \ e^{-C\psi} \sum_{i,j}a^{ij}(h e^{C\psi})_{ij}\\
&=\sum_{i,j}a^{ij}h_{ij}+2C\sum_{i,j}a^{ij}h_i\psi_j
   +Ch\sum_{i,j}a^{ij}\psi_{ij}+C^2h\sum_{i,j}a^{ij}\psi_i\psi_j\\
&\geq -2t\sum_i a^{ii}+2C\sum_{i,j}a^{ij}(-Ch\psi_i)\psi_j+C^2h\sum_{i,j}a^{ij}\psi_i\psi_j\\
&\ \ \
 +Ch\biggl(\langle x,D\psi\rangle
+\varepsilon\sum_{i,j}a^{ij}\psi_i\psi_j
 +\lambda\frac{\langle Df,D\psi\rangle}{\sqrt{1+|Df|^2}}\biggl)\\
&=-\langle x,D h\rangle-2t\sum_1a^{ii}+C(\varepsilon-C)h\sum_{i,j}a^{ij}\psi_i\psi_j
  +Ch\lambda \frac{\langle Df,D\psi\rangle}{\sqrt{1+|Df|^2}}\\
&=2t(|x|^2-\sum_ia^{ii})+C(\varepsilon-C)h\sum_{i,j}a^{ij}\psi_i\psi_j
  -\lambda\sum_i\frac{f_ih_i}{\sqrt{1+|Df|^2}}\\
&=2t\biggl(|x|^2-\sum_ia^{ii}+\lambda\sum_i\frac{f_ix_i}{\sqrt{1+|Df|^2}}\biggl)
 +C(\varepsilon-C)h\sum_{i,j}a^{ij}\psi_i\psi_j\\
&\geq2t\biggl(|x|^2-\sum_ia^{ii}-|\lambda||x|\biggl)
 +C(\varepsilon-C)h\sum_{i,j}a^{ij}\psi_i\psi_j.
\endaligned
\end{equation}
Since  $p\in\{x\in\mathbb{R}^{n}:|x|>R_0\}$,  $R_0\geq R_1$ and
  $C<\varepsilon$, then we obtain from \eqref{eq:2.00} and \eqref{eq:2.2}
\begin{equation}\label{eq:10.01}
\aligned
 0&\geq e^{-C\psi}\sum_{i,j}a^{ij}(h e^{C\psi})_{ij}\\
&\geq2t\biggl(|x|^2-\sum_ia^{ii}-|\lambda||x|\biggl)
 +C(\varepsilon-C)h\sum_{i,j}a^{ij}\psi_i\psi_j>0.
\endaligned
\end{equation}
This is a contradiction. Hence, $h e^{C\psi}$ does not attain its maximum in
 $\{x\in \mathbb{R}^{n}: |x|>R_0\geq R_1\}$.
 Thus,  $\psi$ must be  constant.
\end{proof}

\vskip 2mm
\noindent
{\it Proof of  Theorem \ref{theorem 3}}.
Since $(g_{ij})=(\delta_{ij}+f_if_j)$ is the induced  metric, we know that $(g_{ij})$ a positive definite matrix. Taking $(a_{ij})=(g_{ij})=(\delta_{ij}+f_if_j)$ in the proposition
\ref{lemma 1}
we know $(g_{ij})$ satisfies the condition  \eqref{eq:2.1}.
Putting $\psi=\log\det (g_{ij})$, we have
\begin{equation}
\psi_i=\sum_{p,q}g^{pq}\frac{\partial g_{pq}}{\partial x_i}, \quad \psi_{ij}=\sum_{p,q}(\dfrac{\partial g^{pq}}{\partial x_j}\frac{\partial g_{pq}}{\partial x_i}+g^{pq}\frac{\partial^2 g_{pq}}{\partial x_ix_j}).
\end{equation}
By a direct calculation, we obtain
\begin{equation}
\frac{\partial g^{ip}}{\partial x_l}=-\sum_{j,k}g^{ij}g^{pk}\frac{\partial g_{jk}}{\partial x_l},
\end{equation}

\begin{equation}\label{eq:2.4}
\aligned
\sum_{i,j}g^{ij}\psi_{ij}=&-2\sum_{i,j,p,k,q,l}g^{ij}g^{pk}g^{ql}f_{pi}f_{q}f_{kj}f_l
       -2\sum_{i,j,p,k,q,l}g^{ij}g^{pk}g^{ql}f_{qi}f_pf_{kj}f_l\\
       &+2\sum_{i,j,p,q}g^{ij}g^{pq}f_{pi}f_{qj}+2\sum_{i,j,p,q}g^{ij}g^{pq}f_{pij}f_q.
\endaligned
\end{equation}
From \eqref{eq:2.3}, we get
\begin{equation}
\frac{\partial (\sum_{i,j}g^{ij}f_{ij})}{\partial x_p}=\langle x,Df_p\rangle+\lambda\sum_i\frac{f_if_{ip}}{\sqrt{1+|Df|^2}}.
\end{equation}
Hence, we have
\begin{equation}\label{eq:2.10}
\sum_{i,j}g^{ij}f_{ijp}=\langle x,Df_p\rangle+\lambda\sum_i\frac{f_if_{ip}}{\sqrt{1+|Df|^2}}
+2\sum_{i,j,k,l}g^{ik}g^{jl}f_{kp}f_lf_{ij}.
\end{equation}
From \eqref{eq:2.4} and \eqref{eq:2.10}, we obtain
\begin{equation}
\aligned
\sum_{i,j}g^{ij}\psi_{ij}=&-2\sum_{i,j,p,k,q,l}g^{ij}g^{pk}g^{ql}f_{pi}f_{q}f_{kj}f_l
       +2\sum_{i,j,p,k,q,l}g^{ij}g^{pk}g^{ql}f_{qi}f_pf_{kj}f_l\\
       &+2\sum_{i,j,p,q}g^{ij}g^{pq}f_{pi}f_{qj}+2\sum_{p,q}g^{pq}f_q\langle x,Df_p\rangle\\
       &+2\lambda\sum_{i,p,q}g^{pq}f_q\frac{f_if_{ip}}{\sqrt{1+|Df|^2}}.
\endaligned
\end{equation}
On the other hand, because of
\begin{equation}
-\lambda\frac{\langle Df,D\psi\rangle}{\sqrt{1+|Df|^2}}=-2\lambda\sum_{i,p,q}\frac{g^{pq}f_if_qf_{pi}}{\sqrt{1+|Df|^2}},
\end{equation}
 we have
\begin{equation}\label{eq:2.5}
\aligned
L_{(\lambda, f)}\psi&=\sum_{i,j}g^{ij}\psi_{ij}-\langle x,D\psi\rangle-\lambda\frac{\langle Df,D\psi\rangle}
   {\sqrt{1+|Df|^2}}\\
   &=-2\sum_{i,j,p,k,q,l}g^{ij}g^{pk}g^{ql}f_{pi}f_{q}f_{kj}f_l
       +2\sum_{i,j,p,k,q,l}g^{ij}g^{pk}g^{ql}f_{qi}f_pf_{kj}f_l\\
       &\ \ \  +2\sum_{i,j,p,q}g^{ij}g^{pq}f_{pi}f_{qj}.
\endaligned
\end{equation}
At any fixed point, we can choose a coordinate system $\{x_1,\cdots,x_n\}$ such that
\begin{equation}\label{eq:2.6}
{D f}=(f_1,  0, \cdots, 0).
\end{equation}
 Then we have from \eqref{eq:2.5} and \eqref{eq:2.6}
\begin{equation*}
\aligned
L_{(\lambda, f)}\psi&=-2\sum_{i,p,q}\frac{f_q^2(f_{pi})^2}{(1+f_i^2)(1+f_p^2)(1+f_q^2)}
   +2\sum_{i,p}\frac{(f_{pi})^2}{(1+f_i^2)(1+f_p^2)}\\
  &\ \ \ +2\sum_{i,p,q}\frac{f_pf_qf_{pi}f_{qi}}{(1+f_i^2)(1+f_p^2)(1+f_q^2)}\\
  &=2\sum_i\frac{1}{(1+f_i^2)}\biggl(\frac{f_1f_{1i}}{1+f_1^2}\biggl)^2
  +2\sum_{i,p}\frac{(f_{pi})^2}{(1+f_i^2)(1+f_p^2)(1+f_1^2)}\\
  &=\frac{1}{2}\sum_{i,j}g^{ij}\psi_i\psi_j+2\sum_{i,p}\frac{(f_{pi})^2}{(1+f_i^2)(1+f_p^2)(1+f_1^2)}\\
  &\geq \frac{1}{2}\sum_{i,j}g^{ij}\psi_i\psi_j.
\endaligned
\end{equation*}
Thus, at any point, we have
\begin{equation}\label{eq:2.7}
\aligned
L_{(\lambda, f)}\psi&
  &\geq \frac{1}{2}\sum_{i,j}g^{ij}\psi_i\psi_j.
\endaligned
\end{equation}
From the proposition \ref{lemma 1}, we have $\psi$ is constant. Therefore,  $f_{pi}=0$ for any $p$ and $i$ from \eqref{eq:2.7}. Hence $f$ is a linear function, that is, $X:M\rightarrow \mathbb{R}^{n+1}$ is a hyperplane.
$$
\eqno{\Box}
$$

\section{Proofs of Theorem \ref{theorem 1} and Theorem \ref{theorem 2}}

\noindent
In order to prove the theorem \ref{theorem 1} and the theorem \ref{theorem 2},
we need the following lemma due to Cheng and Wei \cite{cw2}.

\begin{lemma}\label{lemma 2}
Let $X: M\rightarrow \mathbb{R}^{n+1}$ be a complete and non-compact properly immersed
$\lambda$-hypersurface in the Euclidean space $\mathbb{R}^{n+1}$.
Then, there is a positive constant $C$ such that for $r\geq1$,
\begin{equation}
{\rm Area}(B_r(0)\cap X(M))=\int_{B_r(0)\cap X(M)}d\mu\leq Cr^{n+\frac{\lambda^2}2-2\beta-\frac{\inf H^2}2},
\end{equation}
where $B_r(0)$ is a round ball in $\mathbb{R}^{n+1}$ with radius $r$ and centered at the origin,  $\beta=\frac{1}{4}\inf(\lambda-H)^2$.
\end{lemma}

\noindent The next lemma is essentially due to Ding, Xin and Yang \cite{[DXY]}.

\begin{lemma}\label{lemma 3}
Let $X: M\rightarrow \mathbb{R}^{n+1}$ be a complete immersed
$\lambda$-hypersurface in the Euclidean space $\mathbb{R}^{n+1}$.
Then, its Gauss map is a $e^{-\frac{|X|^2}{2}}$-weighted harmonic map.
\end{lemma}

\noindent By using the above two lemmas and the same assertion  as  that of \cite{[DXY]}, we can give the proofs of the theorem \ref{theorem 1}
and the theorem \ref{theorem 2}.

\end {document}